\documentclass[12pt,a4paper,reqno]{article}
\usepackage{amssymb,amsfonts,amsthm,amsmath,amssymb}
\usepackage{caption}
\usepackage{float}
\usepackage{setspace}
\usepackage{epic}
\usepackage{graphics}
\usepackage{graphicx}
 \usepackage[latin2]{inputenc}
\usepackage{hyperref}
\usepackage{enumerate}
\usepackage{indentfirst}
\usepackage[margin=2.9cm]{geometry}
\allowdisplaybreaks
\DeclareMathOperator{\rk}{rk}
\newcommand{\KL}{Kazhdan-Lusztig\ }
\newtheorem{thm}{Theorem}[section]
\newtheorem{lem}[thm]{Lemma}
\newtheorem{prop}[thm]{Proposition}

\newtheorem{conj}[thm]{Conjecture}

\numberwithin{equation}{section}
\usepackage{mma}
\usepackage{ytableau}
\begin{document}
\begin{center}
{\large \bf  The Kazhdan-Lusztig polynomials of uniform matroids}
\end{center}

\begin{center}
Alice L.L. Gao$^{1}$, Linyuan Lu$^2$, Matthew H.Y. Xie$^3$ \\
Arthur L.B. Yang$^{4}$, Philip B. Zhang$^{5}$\\[6pt]

$^{1,3,4}$Center for Combinatorics, LPMC\\
Nankai University, Tianjin 300071, P. R. China\\[6pt]
 
$^{1}$ Department of Applied Mathematics,\\
Northwestern Polytechnical University, Xi'an, Shaanxi 710072, P.R. China\\[6pt]

$^{2}$Department of Mathematics\\
University of South Carolina, Columbia, SC 29208, USA\\[6pt]

$^{5}$ College of Mathematical Science\\
Tianjin Normal University, Tianjin 300387, P. R. China\\[6pt]
Email: $^{1}${\tt llgao@nwpu.edu.cn},
       $^{2}${\tt lu@math.sc.edu}
	   $^{3}${\tt xie@mail.nankai.edu.cn},
       $^{4}${\tt yang@nankai.edu.cn}
       $^{5}${\tt zhang@tjnu.edu.cn}
\end{center}

\begin{abstract}
  The Kazhdan-Lusztig polynomial of a matroid was introduced by Elias, Proudfoot, and Wakefield [{\it Adv. Math. 2016}]. Let $U_{m,d}$ denote the uniform matroid of rank $d$ on a set of $m+d$ elements.
  Gedeon, Proudfoot, and Young [{\it J. Combin. Theory Ser. A, 2017}] pointed out that they can derive
 an explicit formula of the Kazhdan-Lusztig polynomials of $U_{m,d}$ using equivariant Kazhdan-Lusztig polynomials.
In this paper we give two alternative explicit formulas, which allow us to prove the real-rootedness of
the Kazhdan-Lusztig polynomials of $U_{m,d}$ for $2\leq m\leq 15$ and all $d$'s. The case $m=1$ was
previously proved by Gedeon, Proudfoot, and Young [{\it S\'{e}m. Lothar. Combin. 2017}].
We further determine the $Z$-polynomials of all $U_{m,d}$'s
and prove the real-rootedness of the $Z$-polynomials of $U_{m,d}$ for $2\leq m\leq 15$ and all $d$'s.
Our formula also enables us to give an alternative proof of Gedeon, Proudfoot, and  Young's formula for the Kazhdan-Lusztig polynomials of $U_{m,d}$'s without using the equivariant Kazhdan-Lusztig polynomials.
\end{abstract}

\emph{AMS Classification 2010:} 05A15, 26C10, 33F10


\section{Introduction}


 The goal of this paper is threefold: giving a nice formula for computing the Kazhdan-Lusztig polynomials of arbitrary uniform matroids; determining the $Z$-polynomials of arbitrary uniform matroids;
 and proving the real-rootedness of  both the Kazhdan-Lusztig polynomials and the $Z$-polynomials
 for some special uniform matroids.
 Before stating our main results, let us first give an overview of some background.

Recently, Elias, Proudfoot, and Wakefield  \cite{elias2016kazhdan} introduced the notion of the Kazhdan-Lusztig polynomial of a matroid. Given a loopless matroid $M$, let $L(M)$ denote the lattice of flats of $M$, let $\chi_M(t)$ denote its characteristic polynomial, and let $\rk M$ denote the rank of $M$.
They proved that there is a unique way to associate to each $M$ a polynomial $P_M(t) \in \mathbb{Z}[t]$ satisfying the following properties:
\begin{itemize}
\item If $\rk M=0$, then $P_M(t)=1$.

\item If $\rk M>0$, then $\deg P_M(t) < \frac 12 \rk M$.

\item For every $M$, $t^{\rk M}P_M(t^{-1}) = \displaystyle\sum_{F \in L(M)} \chi_{M_F}(t) P_{M^F}(t)$,
\end{itemize}
where the symbol $M^F$ represents the  contraction of $M$ at $F$, and $M_F$ represents the  localization  of $M$ at $F$.

The \KL polynomials for matroids turn out to possess many interesting properties. For example, Elias,  Proudfoot and Wakefield
\cite{elias2016kazhdan} proposed a conjecture which states that the \KL polynomial
of an arbitrary matroid has only non-negative coefficients, and they also proved this conjecture for any representable matroid.
Another interesting conjecture is due to Gedeon, Proudfoot, and  Young \cite{gedeon2016survey}, which states that the \KL polynomial of a matroid has only negative zeros.

To study the properties of the \KL polynomial of a matroid,
it is desirable to give an explicit formula to compute its coefficients. However, it seems hopeless to do this for arbitrary matroid. Recently, much work has been focused on determining the  \KL polynomials for specific families of matroids. For instance, Gedeon \cite{gedeon2016thagomizer} determined the \KL polynomials for thagomizer matroids, Gedeon, Proudfoot, and Young \cite{gedeon2016survey} determined the \KL polynomials for complete bipartite graphs with one part having exactly two vertices, and Lu-Xie-Yang \cite{lu2018kazhdan} determined the \KL polynomials for fan matroids, wheel matroids, and whirl matroids.
The \KL polynomials for braid matroids have been studied in \cite{elias2016kazhdan, gedeon2016survey, karn2018stirling}.

This paper focuses on the \KL polynomials of uniform matroids. Let $U_{m,d}$ denote the uniform matroid of rank $d$ on a set of $m+d$ elements. Throughout this paper, we always assume that $m$ and $d$ are positive integers. Elias, Proudfoot, and Wakefield \cite{elias2016kazhdan} obtained  a recursive relation among the coefficients of \KL polynomials  $P_{U_{m,d}}(t)$. Suppose that
\begin{align*}
 P_{U_{m,d}}(t)=\sum_{i=0}^{\lfloor \frac{d-1}{2}\rfloor}{c_{m,d}^i t^i}.
\end{align*}
Elias, Proudfoot, and Wakefield derived the following result.

\begin{prop}[{\cite[Proposition 2.19 ]{elias2016kazhdan}}]\label{prop-klrec}
For any $m$, $d$, and $0\leq i\leq \lfloor\frac{d-1}{2}\rfloor$, we have
\begin{align}
c_{m,d}^i = (-1)^i\binom{m+d}{i} +
\sum_{j=0}^{i-1}\sum_{k=2j+1}^{i+j}(-1)^{i+j+k}\binom{m+d}{m+k,i+j-k,d-i-j} c^j_{m,k}.\label{uklrec}
\end{align}
\end{prop}

Using this recursion, they obtained explicit formulas for the first few coefficients of $P_{U_{m,d}}(t)$. Based on this  recursive formula, one can also obtain an explicit formula for $P_{U_{1,d}}(t)$, see Proudfoot, Wakefield, and Young \cite{proudfoot2016intersection}.
By introducing the equivariant Kazhdan-Lusztig polynomial of a matroid, Gedeon, Proudfoot, and  Young \cite{gedeon2017equivariant} pointed out that a general formula for $c_{m,d}^i$ can be obtained.
The following explicit formula can be derived using their approach:
\begin{thm}\label{thm-ukl1}
For any $m$, $d$, and $1\leq i\leq \lfloor\frac{d-1}{2}\rfloor$,
we have
\begin{align}\label{ukl1}
c_{m,d}^i=\sum_{h=1}^{\min(m,d-2i)}&
\frac{(e-i-h+1)(m+d)!}{e(e+1)(i+h)(i+h-1)(e-i)!(h-1)!i!(i-1)!},
\end{align}
where  $e=m+d-i-h$.
\end{thm}
In this paper we obtain two alternative formulas for $c_{m,d}^i$. Here is the first one:
\begin{thm}\label{thm-ukl2}
For any $m$, $d$, and $0\leq i\leq \lfloor\frac{d-1}{2}\rfloor$, we have
\begin{align}\label{ukl2}
c_{m,d}^i=\binom{d+m}{i}\sum_{h=1}^{m}\frac{(-1)^{h+1}h}{d-h-i+m} \binom{d-h-i+m}{d-2i-h}\binom{m+i}{m-h}.
\end{align}
\end{thm}

Since the right hand side of \eqref{ukl2} is an alternating sum,
it is hard to deduce the positivity of $c_{m,d}^i$. However, based on this formula we can obtain
another formula for $c_{m,d}^i$, which is manifestly positive.

\begin{thm}\label{thm-ukl3}
For any $m$, $d$, and $0\leq i\leq \lfloor\frac{d-1}{2}\rfloor$, we have
\begin{align}\label{ukl3}
c_{m,d}^i={\frac{1}{d-i} \binom{d+m}{i}
{\sum _{h=0}^{m-1}\binom{d-i+h}{h+i+1}\binom{i-1+h}{h} }}.
\end{align}
\end{thm}

The formula \eqref{ukl3} has some advantages. First, it can be used to
prove \eqref{ukl1} without resorting to the equivariant Kazhdan-Lusztig polynomials of uniform matroids. Secondly, its elegant form allows us to prove the real-rootedness of
the Kazhdan-Lusztig polynomials of some uniform matroids. Gedeon, Proudfoot, and  Young \cite{gedeon2016survey} proved that the polynomial  $P_{U_{1,d}}(t)$ has only negative zeros.
Based on \eqref{ukl3}, we obtain the following result.

\begin{thm}\label{thm-uklroot}
For any $2\leq m\leq 15$ and any $d\geq 1$, the polynomial  $P_{U_{m,d}}(t)$ has only negative zeros.
\end{thm}

The next part of this paper is concerned with the $Z$-polynomials of uniform matroids.
The notion of the $Z$-polynomial of a matroid was introduced by Proudfoot, Xu, and  Young \cite{proudfoot2017z}. Given a matroid $M$, its $Z$-polynomial is defined by
$$Z_M(t):= \sum_{F\in L(M)}{t^ {\rk M_F} P_{M^F}(t)}.$$
Proudfoot, Xu, and  Young \cite{proudfoot2017z}  showed that
\begin{align}\label{kltoz}
 Z_{U_{m,d}}(t)=t^d+\sum_{k=1}^{d}{\binom{d+m}{k+m} t^{d-k}P_{U_{m,k}}(t)}.
\end{align}
Based on this formula, they  proved that $Z_{U_{1,d}}(t)$ is just a Narayana polynomial. Denote by $z_{m,d}^i$ the coefficient of $t^i$ in $Z_{U_{m,d}}(t)$. We obtain an explicit expression  of $z_{m,d}^i$ as given below.

\begin{thm}\label{thm-uz1}
For any $m$, $d$, and $0\leq i\leq d$,
we have
\begin{align}
z_{m,d}^i=\frac{\binom{d+m}{i+m} \binom{d+m}{i}}{\binom{d+m}{m}}
{\sum _{h=0}^{m-1}  \frac{i (h-m+1)+m}{(h+1) m} \binom{i-1+h}{h} \binom{d-i+h}{h}}\label{uz1}
\end{align}
\end{thm}

Proudfoot, Xu, and  Young \cite{proudfoot2017z} also conjectured that the $Z$-polynomial $Z_M(t)$ has only negative zeros for any matroid $M$. It is well known that the classical Narayana polynomial has only negative zeros. Thus, their conjecture is valid for $Z_{U_{1,d}}(t)$.
Parallel to Theorem \ref{thm-uklroot}, we obtain the following result.

\begin{thm}\label{thm-uzroot}
For $2\leq m\leq 15$ and any $d\geq 1$, the polynomial $Z_{U_{m,d}}(t)$ has only negative zeros.
\end{thm}

This paper is organized as follows. In Section \ref{sec:unewkl} we first give a proof of Theorem \ref{thm-ukl2} by using Proposition \ref{prop-klrec}, and then derive Theorem \ref{thm-ukl3}
from Theorem \ref{thm-ukl2}. We would like to point out that Zeilberger's algorithm plays an important role for our proofs of Theorems \ref{thm-ukl2} and \ref{thm-ukl3}.
The second part of Section \ref{sec:unewkl} is devoted to the proof of Theorem \ref{thm-uklroot}. Finally we give a new proof of Theorem \ref{thm-ukl1} without the help of equivariant Kazhdan-Lusztig polynomials. In Section \ref{sec:uz}, we prove Theorems \ref{thm-uz1} and \ref{thm-uzroot}.
To prove the real-rootedness of $P_{U_{m,d}}(t)$ and $Z_{U_{m,d}}(t)$, we utilize the theory of multiplier sequences and the theory of $n$-sequences.

\section{The \KL polynomials}\label{sec:unewkl}

This section is devoted to the study of the \KL polynomials of uniform matroids.
First, we  verify that \eqref{ukl2} satisfies the recursive relation \eqref{uklrec}, and then derive \eqref{ukl3} from \eqref{ukl2}.
Secondly, we  use \eqref{ukl2} to prove Theorem \ref{thm-uklroot}. Finally, we  show how to prove \eqref{ukl1} by using \eqref{ukl2}.

\subsection{Polynomial coefficients}

The aim of this subsection is to prove Theorems \ref{thm-ukl2} and \ref{thm-ukl3}.


\begin{proof}[Proof of Theorem \ref{thm-ukl2}]
It suffices to show that \eqref{ukl2} satisfies the recursion \eqref{uklrec} together with the initial values $c_{m,1}^0=1$. It is straightforward to verify that the right hand side of \eqref{ukl2} is equal to $1$ when $d=1$ and $i=0$.

It remains to show that \eqref{ukl2} satisfies the recursion \eqref{uklrec}. To this end, we substitute \eqref{ukl2} into \eqref{uklrec}, which yields the left hand side
\begin{align}
(LHS)&=\binom{d+m}{i}\sum_{h=1}^{m}\frac{(-1)^{h+1}h}{d-h-i+m} \binom{d-h-i+m}{d-2i-h}\binom{m+i}{m-h}\label{ukl2-lhs}
\end{align}
and the right hand side
\begin{align*}
(RHS)=(-1)^i\binom{d+m}{i}+&\sum_{j=0}^{i-1}\sum_{k=2j+1}^{i+j}
\sum_{h=1}^{m}(-1)^{i+j+k+h+1}\frac{h}{k-h-j+m}\binom{m+j}{m-h}
\binom{k+m}{j} \notag\\
&\qquad \qquad \quad \times  \binom{k-h-j+m}{k-2j-h}\binom{m+d}{m+k,i+j-k,d-i-j}.
\end{align*}
It is enough to show that $(LHS)=(RHS)$.

In the following we will reduce the triple summation in $(RHS)$ into a single summation. By interchanging the order of summation of $(RHS)$, we obtain
\begin{align*}
(RHS)=(-1)^i\binom{d+m}{i}+&\sum_{j=0}^{i-1}
\sum_{h=1}^{m}\sum_{k=2j+1}^{i+j}(-1)^{i+j+k+h+1}\frac{h}{k-h-j+m}\binom{m+j}{m-h}
\binom{k+m}{j} \notag\\
&\qquad \qquad \quad \times  \binom{k-h-j+m}{k-2j-h}\binom{m+d}{m+k,i+j-k,d-i-j}.
\end{align*}
Note that
\begin{align*}
\binom{m+d}{m+k,i+j-k,d-i-j}&=\binom{m+d}{d-i-j}\binom{m+i+j}{m+k},\\[5pt]
\binom{k-h-j+m}{k-2j-h}&=\frac{k-h-j+m}{m+j}\binom{k-h-j+m-1}{k-2j-h}.
\end{align*}
Substituting into the right hand side of the above summation, we get
\begin{align*}
(RHS)=(-1)^i\binom{d+m}{i}+&\sum_{j=0}^{i-1}
\sum_{h=1}^{m}\sum_{k=2j+1}^{i+j}(-1)^{i+j+k+h+1}\frac{h}{m+j}\binom{m+j}{m-h}
\binom{k+m}{j} \notag\\
&\times  \binom{k-h-j+m-1}{k-2j-h}\binom{m+d}{d-i-j}\binom{m+i+j}{m+k}.
\end{align*}
Therefore,
\begin{align}
(RHS)&=(-1)^i\binom{d+m}{i}+\sum_{j=0}^{i-1}
\sum_{h=1}^{m}(-1)^{i+j+h+1}\frac{h}{m+j}\binom{m+j}{m-h}\binom{m+d}{d-i-j}F_{j,h},\label{ukl2-rhs}
\end{align}
where
\begin{align*}
F_{j,h}&=\sum_{k=2j+1}^{i+j}(-1)^k
\binom{k+m}{j}\binom{k-h-j+m-1}{k-2j-h}\binom{m+i+j}{m+k}.
\end{align*}

We claim that
\begin{align}
F_{j,h}=(-1)^h\binom{m+i+j}{m+i}\binom{i-j}{h}.\label{eq-mid}
\end{align}
This is because
\begin{align*}
F_{j,h}&=\sum_{k=2j+h}^{i+j}(-1)^k
\binom{k+m}{j}\binom{k-h-j+m-1}{k-2j-h}\binom{m+i+j}{m+k}\\[5pt]
&=\sum_{k=0}^{i-j-h}(-1)^{k+h}
\binom{k+2j+h+m}{j}\binom{k+j+m-1}{k}\binom{m+i+j}{m+k+2j+h}\\[5pt]
&=(-1)^h\sum_{k=0}^{i-j-h}(-1)^{k}\binom{k+j+m-1}{k}
\binom{m+i+j}{m+k+2j+h}\binom{k+2j+h+m}{j}\\[5pt]
&=(-1)^h\sum_{k=0}^{i-j-h}(-1)^{k}\binom{k+j+m-1}{k}
\binom{m+i+j}{m+i}\binom{m+i}{i-j-h-k}\\[5pt]
&=(-1)^h\binom{m+i+j}{m+i}\sum_{k=0}^{i-j-h}\binom{-j-m}{k}
\binom{m+i}{i-j-h-k}\\[5pt]
&=(-1)^h\binom{m+i+j}{m+i}\binom{i-j}{i-j-h},
\end{align*}
where the last equality is obtained by the Chu-Vandermonde identity.

Substituting \eqref{eq-mid} into \eqref{ukl2-rhs}, we obtain that \begin{align*}
(RHS)=&(-1)^i\binom{d+m}{i}+\sum_{j=0}^{i-1}
(-1)^{i+j+1}\binom{m+d}{d-i-j}\binom{m+i+j}{m+i}\\
& \times \left(\sum_{h=1}^{m}\frac{h}{m+j}\binom{m+j}{m-h}
\binom{i-j}{h}\right).
\end{align*}
Again by the Chu-Vandermonde identity, we have
$$\sum_{h=1}^{m}\frac{h}{m+j}\binom{m+j}{m-h}
\binom{i-j}{h}=\sum_{h=1}^{m}\frac{i-j}{m+j}\binom{m+j}{m-h}
\binom{i-j-1}{h-1}=\frac{i-j}{m+j}\binom{m+i-1}{m-1}.$$
Thus, we have
\begin{align*}
(RHS)=&(-1)^i\binom{d+m}{i}+\sum_{j=0}^{i-1}
(-1)^{i+j+1}\frac{i-j}{m+j}\binom{m+d}{d-i-j}\binom{m+i+j}{m+i}\binom{m+i-1}{m-1}.
\end{align*}
Combining the above identity and \eqref{ukl2-lhs}, we see that
$(LHS)=(RHS)$ is equivalent to  the following identity:
\begin{align*}
&\sum_{h=1}^m \frac{(-1)^{i+h+1} h (d-h-i+m-1)!}{(h+i)! (m-h)! (d-2 i-h)!}-\sum_{j=0}^{i}\frac{(-1)^{j+1}(i-j)(m+d-i)!}{(i+m)(j+m)j!(d-i-j)!(m-1)!}=1.
\end{align*}
It remains to prove the above identity.
Let
\begin{align*}
 p_m&=\sum_{h=1}^m \frac{(-1)^{i+h+1} h (d-h-i+m-1)!}{(h+i)! (m-h)! (d-2 i-h)!},\\
 q_m&=\sum_{j=0}^{i}\frac{(-1)^{j+1}(i-j)(m+d-i)!}{(i+m)(j+m)j!(d-i-j)!(m-1)!}.
\end{align*}
Since both $p_m$ and $q_m$ are hypergeometric summations, we are able to prove  $p_m-q_m=1$ with the aid of a computer algebra system.
As illustrated by the following lines, the application of Zeilberger's algorithm yields the following equality
$$p_{m+1}-p_{m}=q_{m+1}-q_{m}.$$
Here we use a Mathematica package \textit{fastZeil} due to Paule and Schorn \cite{paule1995mathematica}.

\begin{mma}
\In <<|RISC| ~\grave{} |fastZeil|~\grave{};\\
\end{mma}
\begin{mma}
\In |Zb|\big[\frac{(-1)^{j+1} (i-j) (d-i+m)!}{(i+m) (j+m) j! (d-i-j)! (m-1)! },\{j,0,i\},m\big];\\
\end{mma}
\begin{mma}
\In |FullSimplify|[\%]/.|Gamma|[n\_]\to(n-1)!\\
\end{mma}
\begin{mma}
\Out \left\{(d-i) (|SUM|[m]-|SUM|[1+m])==\left.\frac{(-1)^{i+1} (d-i+m)!}{(i+m) (i+m+1) (i-1)! m!  (d-2 i-1)!}\right\}\right\}\\
\end{mma}
\begin{mma}
\In |Zb|\big[\frac{h (-1)^{i+h+1} (d-h-i+m-1)!}{(h+i)! (m-h)! (d-2 i-h)!},\{h,0,m\},m\big];\\
\end{mma}
\begin{mma}
\In  |FullSimplify|[\%]/.|Gamma|[n\_]\to(n-1)!\\
\end{mma}
\begin{mma}
\Out \left\{(d-i) (i+m) (i+m+1) (-|SUM|[m]+|SUM|[1+m])==\frac{(-1)^i (d-i+m)!}{(i-1)! m! (d-2 i-1)!}\right\}\\
\end{mma}

The proof of the theorem will be complete once we show that $p_{1}-q_{1}=1$. By direct computation, we have
\begin{align*}
p_1= \frac{(-1)^{i}}{i+1}\binom{d-i-1}{i}.
\end{align*}
On the other hand, we have
\begin{align*}
q_1&=\frac{1}{i+1}\times \sum_{j=0}^{i-1}(-1)^{j+1}(i-j)\binom{d-i+1}{j+1}\\[5pt]
&=-1+\frac{1}{i+1}\times \sum_{j=-1}^{i-1}(-1)^{j+1}(i-j)\binom{d-i+1}{j+1}\\[5pt]
&=-1+\frac{(-1)^i}{i+1}\times \sum_{j=0}^i(-1)^{i-j}(i-j+1)\binom{d-i+1}{j}\\
&=-1+\frac{(-1)^i}{i+1}\times \sum_{j=0}^i\binom{-2}{i-j}\binom{d-i+1}{j}\\
&=-1+\frac{(-1)^i}{i+1}\binom{d-i-1}{i},
\end{align*}
where the last equality is obtained by the Chu-Vandermonde identity.
This completes the proof.
\end{proof}

\begin{proof}[Proof of Theorem \ref{thm-ukl3}]
In view of \eqref{ukl2} and \eqref{ukl3}, it suffices to show that
\begin{align*}
\sum_{h=1}^{m}\frac{(-1)^{h+1}h}{d-h-i+m} \binom{d-h-i+m}{d-2i-h}\binom{m+i}{m-h}={\frac{1}{d-i}{\sum _{h=0}^{m-1}\binom{d-i+h}{h+i+1}\binom{i-1+h}{h} } }.
\end{align*}
Denote by $f_{m,d}^i$ the left hand side of the above identity,
and denote by $g_{m,d}^i$ its right hand side.
It is routine to show that
\begin{align*}
f_{1,d}^i=g_{1,d}^i= \frac{1}{d-i}\binom{d-i}{i+1}.
\end{align*}
Therefore, it is sufficient to show that
\begin{align*}
f_{m+1,d}^i-f_{m,d}^i=g_{m+1,d}^i-g_{m,d}^i={\frac{1}{d-i}
\binom{d-i+m}{m+i+1}\binom{i-1+m}{m} }.
\end{align*}
Now apply Zeilberger's algorithm to $f_{m,d}^i$ along the following lines.

\begin{mma}
\In |Zb|\big[\frac{(-1)^{h+1} h}{d-h-i+m}|Binomial|[d-h-i+m,d-h-2 i] |Binomial|[i+m,m-h]\break  ,\{h,1,m\},m,1\big];\\
\begin{mma}
\end{mma}
\In |FullSimplify|[\%]/.|Gamma|[n\_]\to(n-1)!\\
\Out \left\{(d - i) (i + m) (1 + i + m)(-|SUM|[m]+|SUM|[1+m])==\frac{(d-i+m)!}{(i-1)! m! (d-2 i-1)!}\right\}\\
\end{mma}

Thus we have
\begin{align*}
f_{m+1,d}^i-f_{m,d}^i&=\frac{ (d-i+m)!}{(d-i)(i+m)(i+m+1) (i-1)!m!(d-2 i-1)!}\\[5pt]
&={\frac{1}{d-i}\binom{d-i+m}{m+i+1}\binom{i-1+m}{m} },
\end{align*}
as desired.
This completes the proof.
 \end{proof}

\subsection{Real zeros}\label{sec:uklroot}

This subsection is devoted to the study of the real-rootedness of the Kazhdan-Lusztig polynomials $P_{U_{m,d}}(t)$ by using the  theory of multiplier sequences and the theory of $n$-sequences, for which we refer the reader to \cite{craven1983location1,craven1980intersections,craven2004composition}.

Let us recall some related concepts and results.
A sequence $\Gamma=\{\gamma_k\}_{k=0}^{\infty}$ of real numbers  is called a \emph{multiplier sequence} if, whenever any  real polynomial
$$f(t)=\sum_{k=0}^{n}a_kt^k$$ has only real zeros, so does the polynomial $$\Gamma[f(t)]=\sum_{k=0}^{n}\gamma_ka_kt^k.$$
We have the following result.

\begin{lem}\label{ms-u}
For any $m$ and $d$, the sequence  $\{\binom{d+2 m}{i+m}\}_{i=0}^{\infty}$
is a multiplier sequence.
\end{lem}

\begin{proof}
It immediately follows from the following known fact:

both
$\{\frac{1}{(m+i)!}\}_{i=0}^{\infty}$ and $\{\frac{1}{(m+d-i)!}\}_{i=0}^{\infty}$ are multiplier sequences,
see \cite{lu2018kazhdan} and references therein.
\end{proof}

A sequence $\Gamma=\{\gamma_i\}_{k=0}^{n}$  is called an \emph{$n$-sequence} if for every polynomial $f(t)$ of degree less than or equal to $n$ and with only real zeros, the polynomial $\Gamma[f(t)]$ also has only real zeros.
We need the following algebraic characterization of $n$-sequences.

\begin{thm}[\cite{craven1983location1}]\label{n-char}
Let $\Gamma=\{\gamma_k\}_{k=0}^{n}$ be a sequence of  real numbers. Then $\Gamma$ is an $n$-sequence if and only if the zeros of the polynomial
$\Gamma[(1+t)^n]$ are all real and of the same sign.
\end{thm}

Now we can prove Theorem \ref{thm-uklroot}. By Theorem \ref{thm-ukl3}, we see that
\begin{align}\label{eq-ukl3-reform}
P_{U_{m,d}}(t)=\frac{1}{\binom{d+2 m}{m}}\sum_{i=0}^{\lfloor \frac{d-1}{2}\rfloor} {
\binom{d+2 m}{i+m}\binom{d-i-1}{i} f_{m}(d,i)t^i},
\end{align}
where
\begin{align}
f_{m}(d,i)=\sum _{h=0}^{m-1} \frac{1}{(m-h) \binom{m}{h}} {\binom{i+m}{m-h-1} \binom{i-1+h}{h} \binom{d-i+h}{h}}.\label{eq-klpoldi}
\end{align}
Note that $i$ can take any integer value between $0$ and $d$ in the above formula, namely, $f_{m}(d,i)$ is well defined for $0\leq i\leq d$. Moreover, it is straightforward to compute that
\begin{align*}
f_{m}(d,d)&=\sum _{h=0}^{m-1} \frac{1}{(m-h) \binom{m}{h}} {\binom{d+m}{m-h-1} \binom{d-1+h}{h}}\\
&=\sum _{h=0}^{m-1} \frac{h!(m-h-1)!}{m!}\cdot\frac{(m+d)!}{(m-h-1)!(d+h+1)!}
\cdot\frac{(d+h-1)!}{(d-1)!h!}\\
&=\frac{(m+d)!}{m!(d-1)!}\sum _{h=0}^{m-1} \frac{1}{(d+h)(d+h+1)}\\
&=\frac{(m+d)!}{m!(d-1)!}\left(\frac{1}{d}-\frac{1}{d+m}\right)\\
&=\frac{(m+d)!}{m!(d-1)!}\cdot\frac{m}{d(d+m)}\\
&=\binom{m+d-1}{m-1}.
\end{align*}

To prove Theorem \ref{thm-uklroot}, we first establish the following result.

\begin{lem}\label{lem-uklroots}
For any  $m$ and $d$, the polynomial
\begin{align}
\frac{1}{\binom{d+2 m}{m}}\sum_{i=0}^{\lfloor \frac{d-1}{2}\rfloor} {
\binom{d+2 m}{i+m}\binom{d-i-1}{i}  t^i},
\end{align}
has only real zeros.
\end{lem}

\begin{proof}
By Lemma \ref{ms-u} it suffices to show that
\begin{align}
\sum_{i=0}^{\lfloor \frac{d-1}{2}\rfloor} {
\binom{d-i-1}{i}  t^i},
\end{align}
has only real zeros. This is true since
the $d$-th  Fibonacci  polynomial
 $$F_d(t)=\sum_{i=0}^{\lfloor \frac{d-1}{2} \rfloor} \binom{d-i-1}{i} t^{d-2i-1}$$
has only pure imaginary zeros for $d\geq 3$,
 see \cite{hoggatt1973roots,lachal2013trick}.
\end{proof}

In view of \eqref{eq-ukl3-reform} and Lemma \ref{lem-uklroots},
Theorem \ref{thm-uklroot} will be proved once we show that $\{f_{m}(d,i)\}_{i=0}^{d}$ is a $d$-sequence for $2\leq m\leq 15$. Numerical evidence suggests the following conjecture.

\begin{conj}\label{conj-ukl}
 For any $m,d$, and $0\leq i\leq d$, let $f_{m}(d,i)$ be defined as in \eqref{eq-klpoldi}. Then $\{f_{m}(d,i)\}_{i=0}^{d}$ is a $d$-sequence.
\end{conj}

Note that if Conjecture \ref{conj-ukl} is true, then the real-rootedness of the \KL polynomials $P_{U_{m,d}}(t)$ will immediately follow from Lemma \ref{lem-uklroots}. While Conjecture \ref{conj-ukl} is generally open, we can prove the following result.

\begin{thm}\label{thm-dsequence}
For any $d$ and $2\leq m\leq 15$, the sequence $\{f_{m}(d,i)\}_{i=0}^{d}$ is a $d$-sequence.
\end{thm}

Next we will provide a general approach to prove Theorem \ref{thm-dsequence}. Theoretically, our approach is applicable for any given $m$. However, we would need more and more computer time as $m$ becomes larger. (On a  Ubuntu  server  with  Intel  Xeon  CPU  E5-2640 v3 (2.60GHz), it takes 9 hours calculation for $m=14$ and 20 hours for $m=15$.)

By Theorem \ref{n-char}, in order to prove \ref{thm-dsequence}, it suffices to show that
\begin{align}\label{eq-realroots}
Q_d(t)=\sum_{i=0}^{d} f_{m}(d,i) \binom{d}{i} t^i
\end{align}
has only nonnegative zeros for $2\leq m\leq 15$.
By \eqref{eq-klpoldi} it is easy to see that $f_{m}(d,i)$ can be considered as a polynomial in $i$ and $d$ for any fixed $m$, and
as a polynomial in $i$ the polynomial $f_{m}(d,i)$ is of degree
$2(m-1)$ with leading coefficients $\frac{(-1)^{m-1}}{(m-1)! m!}$ and constant term $1$. Thus, we may express $f_{m}(d,i)$ in terms of the falling factorials, say
\begin{align}\label{eq-fg}
f_{m}(d,i)=\sum_{k=0}^{2(m-1)} g_{m,k}(d) (i)_k,
\end{align}
where $g_{m,k}(d)$  is a polynomial of $d$ and
$(i)_k=i(i-1)\cdots(i-k+1)$.
It is easy to see that $g_{m,0}(d)=1$ and $g_{m,2(m-1)}(d)=\frac{(-1)^{m-1}}{(m-1)! m!}$.
Combining \eqref{eq-realroots} and \eqref{eq-fg},
we obtain that
\begin{align*}
Q_d(t)&=\sum_{i=0}^{d} \left(\sum_{k=0}^{2(m-1)} g_{m,k}(d)  (i)_k \right)\binom{d}{i} t^i\\
&=\sum_{k=0}^{2(m-1)} g_{m,k}(d) \left(\sum_{i=0}^{d}  (i)_k   \binom{d}{i} t^i\right)\\
&=\sum_{k=0}^{2(m-1)} g_{m,k}(d) t^k  ( (1+t)^d)^{(k)}\\
&=\sum_{k=0}^{2(m-1)} g_{m,k}(d)   (d)_k    t^k (1+t)^{d-k}\\
&=(1+t)^d \sum_{k=0}^{2(m-1)} g_{m,k}(d) (d)_k \left(\frac{t}{1+t}\right)^k.
\end{align*}
Therefore, to prove Theorem \ref{thm-dsequence} it suffices to show that
\begin{align}\label{eq-gdt}
G_{m,d}(t)=\sum_{k=0}^{2(m-1)} g_{m,k}(d) (d)_k t^k
\end{align}
has only real zeros for $2\leq m\leq 15$. Note that
$G_{m,d}(1)\neq 0$ since, by \eqref{eq-fg}, we have
$$G_{m,d}(1)=\sum_{k=0}^{2(m-1)} g_{m,k}(d) (d)_k 1^k=f_{m}(d,d)=\binom{d+m-1}{m-1}>0.$$

We proceed to show our strategy to prove the real-rootedness of
$G_{m,d}(t)$ defined by \eqref{eq-gdt}. For given $m$ and $d<2(m-1)$,
we can directly  verify the real-rootedness of $G_{m,d}(t)$ when $d<2(m-1)$ with the  help of a computer algebra system. When $d\geq 2(m-1)$ for fixed $m$, we  will prove that $G_{m,d}(t)$ is of degree $2(m-1)$ and it has $2(m-1)$ distinct real zeros. The former is obvious, and
the latter can be proved via a criterion  due to Borchardt and  Hermite, we will recall below, see the discussions in \cite[pp. 349]{Rahman2002Analytic}.  Suppose that
$A(t)=\sum_{i=0}^{n}a_{n-i}t^{i}$ and $B(t)=\sum_{i=0}^{n}b_{n-i}t^{i}$
are two polynomials with $a_0\ne 0$. For any $1\leq k\leq n$, let
\begin{align*}
  \Delta_{2 k}\left( A(t),B(t) \right)= \det \left(
  \begin{array}{ccccc}
   a_0 & a_1 & a_2 & \dots & a_{2k-1}\\
   b_0 & b_1 & b_2 & \dots & b_{2k-1}\\
    0  & a_0 & a_1 & \dots & a_{2k-2}\\
    0  & b_0 & b_1 & \dots & b_{2k-2}\\
  \vdots&\vdots&\vdots&  & \vdots\\
    0  &  0  &  0  & \dots & b_{k}\\
  \end{array}\right)_{2k \times 2k}.
\end{align*}
These determinants are known as the Hurwitz determinants of $A(t)$ and $B(t)$. Borchardt and  Hermite's criterion can be stated as follows.

\begin{thm}[{\cite[Corollary 10.6.13]{Rahman2002Analytic}}]\label{rz-criterion}
Suppose that $A(t)$ is a real polynomial of degree $n$ with $a_0\neq 0$. Then
$A(t)$ has $n$ distinct real zeros  if and only if the corresponding Hurwitz determinants satisfy
\begin{align}\label{Hurwitz_det}
\Delta_{2k}(A(t),A'(t))>0, \quad \text{for every } 1\leq k \leq n.
\end{align}
\end{thm}

Now we can prove Theorem \ref{thm-dsequence}.

\begin{proof}[Proof of Theorem \ref{thm-dsequence}.]
Following the preceding arguments, we will take $m=2$ to illustrate our proof.
We first expand $f_{m}(d,i)$ in terms of the falling factorials.
For $m=2$ we have
$$f_{m}(d,i)=1+\frac{(d+2)i}{2}-\frac{i^2}{2}=(i)_0+\frac{d+1}{2}(i)_1-\frac{1}{2}(i)_2.$$
Then we determine the explicit expression of $G_{m,d}(t)$ defined by \eqref{eq-gdt}. In the case of $m=2$, we have
$$G_{m,d}(t)=1+\frac{d+1}{2} (d)_1t-\frac{1}{2}(d)_2t^2=1+\frac{d(d+1)}{2}t-\frac{d(d-1)}{2}t^2.$$
Finally, it remains to prove the real-rootedness of $G_{m,d}(t)$
under each of the following two cases: (i) $d<2(m-1)$; (ii) $d\geq 2(m-1)$.

Let us first consider the case of $d<2(m-1)$.
For $m=2$, such $d$ can only be $1$. In this case, we have
$G_{2,1}(t)=1+t$, which is clearly real-rooted.

We proceed to consider the case of $d\geq 2(m-1)$.
In this case, $G_{m,d}(t)$ is a polynomial in $t$ of degree $2(m-1)$.
By Theorem \ref{rz-criterion}, we need to prove the positivity of $2(m-1)$ Hurwitz determinants $\Delta_{2k}(G_{m,d}(t),G'_{m,d}(t))$ for $1\leq k\leq 2(m-1)$.
To this end, let $d'=d-2(m-1)\geq 0$. It suffices to show that for any $1\leq k\leq 2(m-1)$  the Hurwitz determinant $\Delta_{2k}(G_{m,d}(t),G'_{m,d}(t))$ is a polynomial in $d'$ with positive coefficients. For $m=2$, it is straightforward to compute that
\begin{align*}
\Delta_{2}(G_{2,d}(t),G'_{2,d}(t))&=\det\left(
\begin{array}{cc}
 -\frac{1}{2} (d-1) d & \frac{1}{2} d (d+1) \\
 0 & -(d-1) d \\
\end{array}
\right)\\
&=\frac{1}{2} (d-1)^2 d^2\\
&=\frac{d'^4}{2}+3 d'^3+\frac{13 d'^2}{2}+6 d'+2,
\end{align*}
and
\begin{align*}\Delta_{4}(G_{2,d}(t),G'_{2,d}(t))&=\det \left(
\begin{array}{cccc}
 -\frac{1}{2} (d-1) d & \frac{1}{2} d (d+1) & 1 & 0 \\
 0 & -(d-1) d & \frac{1}{2} d (d+1) & 0 \\
 0 & -\frac{1}{2} (d-1) d & \frac{1}{2} d (d+1) & 1 \\
 0 & 0 & -(d-1) d & \frac{1}{2} d (d+1) \\
\end{array}
\right)\\
&=\frac{1}{16} (d-1)^2 d^3 \left(d^3+2 d^2+9 d-8\right)\\
&=\frac{d'^8}{16}+d'^7+\frac{59 d'^6}{8}+31 d'^5+\frac{1265 d'^4}{16}+124 d'^3+\frac{233 d'^2}{2}+60 d'+13,
\end{align*}
where $d'=d-2\geq 0$.
For larger $m$, both computing the Hurwitz determinants and verifying the positivity of polynomial coefficients can be done with the help of a computer algebra system. For $3\leq m \leq 15$, a similar calculation can be found in \linebreak
\href{https://github.com/mathxie/kl_uniform_matroid}{https://github.com/mathxie/kl\_uniform\_matroid}.
This completes the proof.
\end{proof}

Finally, we are in the position to prove Theorem \ref{thm-uklroot}.

\begin{proof}[Proof of Theorem \ref{thm-uklroot}.]
It immediately follows from Lemma \ref{lem-uklroots} and Theorem \ref{thm-dsequence} in view of the positivity of polynomial coefficients.
\end{proof}

\subsection{Gedeon, Proudfoot and  Young's formula}
The aim of this subsection is to prove  Gedeon, Proudfoot, and  Young's formula \eqref{ukl1}
based on our new formula \eqref{ukl3} for $c_{m,d}^i$.

To be self-contained, we will first recall Gedeon, Proudfoot, and  Young's original proof of \eqref{ukl1}, with more details added here. Given a nonnegative integer $n$, a partition of $n$ is a tuple $\lambda=(\lambda_1,\lambda_2,\ldots,\lambda_k)$ of nonnegative integers such that $\lambda_1\geq\lambda_2\geq\cdots\geq\lambda_k \geq 0$ and $\sum_{i=1}^k\lambda_i=n$, denoted by $\lambda\vdash n$. For any partition $\lambda\vdash n$, let $\ell(\lambda)$ denote the number of its nonzero parts and let $V[\lambda]$ be the irreducible representation of $S_{n}$ indexed by $\lambda$.  Gedeon, Proudfoot, and  Young \cite{gedeon2017equivariant}  obtained the following result.
\begin{thm}[{\cite[Theorem 3.1]{gedeon2017equivariant}}]\label{thm-ueqkl}
For all positive $m$, $d$, and $i$, we have
\begin{align}\label{um_1}
c_{m,d}^i=\sum_{h=1}^{\min(m,d-2i)}\dim V[d+m-2i-h+1,h+1,2^{i-1}].
\end{align}
\end{thm}

To derive \eqref{ukl1} from the above theorem, we need to compute $\dim V[d+m-2i-h+1,h+1,2^{i-1}]$ by using the hook-length formula. Let us recall some related definitions.
Each partition $\lambda$ is associated to a left justified array of squares with $\lambda_i$ cells in the $i$-th row, called the Young diagram of $\lambda$. The square in the $i$-th row and $j$-th column is denoted by $(i, j)$.
The hook-length of $(i,j)$, denoted by $h(i,j)$, is defined
to be the number of cells directly to the right or directly below $(i,j)$, counting $(i,j)$ itself once. The well known hook-length formula states that
\begin{align}\label{hook_length}
\dim V[\lambda]=\frac{n!}{\displaystyle{\prod_{ \genfrac{}{}{0pt}{}{1\leq i \leq \ell(\lambda)}  {1\leq j \leq \lambda_i}}}h_{(i,j)}},
\end{align}
see \cite[pp. 50]{fulton2013representation}.

Now one can prove \eqref{ukl1} by using \eqref{um_1} and \eqref{hook_length}.

\begin{proof}[The first proof of Theorem \ref{thm-ukl1}]
It suffices to show that $\dim V[d+m-2i-h+1,h+1,2^{i-1}]$ is equal to
\begin{align*}
\frac{(e-i-h+1)(m+d)!}{e(e+1)(i+h)(i+h-1)(e-i)!(h-1)!i!(i-1)!},
\end{align*}
where $e=m+d-i-h$.
We can check that the hook-lengths of the first row of the partition $(d+m-2i-h+1,h+1,2^{i-1})$ are given by
\begin{align}
h_{(1,j)}=\left\{
\begin{array}{ll}
m+d-i-h+2-j, & \mbox{for $1\leq j\leq 2$;}\\
m+d-2i-h+3-j, & \mbox{for $3\leq j\leq h+1$;}\\
m+d-2i-h+2-j, & \mbox{for $h+2\leq j\leq m+d-2i-h+1$,}
\end{array}
\right.\label{eq-hooklengths}
\end{align}
%
%
the hook-lengths of the second row of the partition $(d+m-2i-h+1,h+1,2^{i-1})$ are given by
\begin{align}
h_{(2,j)}=\left\{
\begin{array}{ll}
h+i-j+1, & \mbox{for $1\leq j\leq 2$;}\\
h+2-j, & \mbox{for $3\leq j\leq h+1$,}
\end{array}
\right.
\end{align}
and the hook-lengths of the last $i-1$ rows of the partition $(d+m-2i-h+1,h+1,2^{i-1})$ are given by
\begin{align*}
h_{(k,j)}=i-k-j+4 \mbox{\quad for $3\leq k\leq i+1$ and $1\leq j\leq 2$.}
\end{align*}
It is routine to verify that
\begin{align}
\displaystyle{\prod_{1\leq j\leq m+d-2i-h+1}}h_{(1,j)}
&=(m+d-i-h)(m+d-i-h+1)\times \frac{(m+d-2i-h)!}{(m+d-2i-2h+1)}\label{u-hooks-1};\\
\displaystyle{\prod_{1\leq j\leq b+1}}h_{(2,j)}&=(i+h)(i+h-1)\times (h-1)!\label{u-hooks-2};\\
\displaystyle{\prod_{3\leq k \leq i+1, 1\leq j\leq 2}}h_{(k,j)}&=i!(i-1)!.\label{u-hooks-3}
\end{align}
Therefore, by \eqref{hook_length},
$\dim V[d+m-2i-h+1,h+1,2^{i-1}]$ is equal to
\begin{align*}
\frac{(m+d)!}{\displaystyle{\prod_{1\leq j\leq m+d-2i-b+1}}h_{(1,j)}\times \displaystyle{\prod_{1\leq j\leq b+1}}h_{(2,j)} \times \displaystyle{\prod_{3\leq k \leq i+1, 1\leq j\leq 2}}h_{(k,j)}}.
\end{align*}
Substituting \eqref{u-hooks-1},\eqref{u-hooks-2} and  \eqref{u-hooks-3} into the above identity, we obtain the desired result. This completes the proof.
\end{proof}

As remarked by Gedeon, Proudfoot, and  Young, after they figured out the formula \eqref{ukl1}, they attempted to prove this formula directly but failed. In the following we shall use
\eqref{ukl3} to prove \eqref{ukl1}. Since \eqref{ukl3} is derived
from \eqref{ukl2} while \eqref{ukl2} can be proved by \eqref{uklrec}, the following proof provides a direct way to prove \eqref{ukl1}.

\begin{proof}[The second proof of Theorem \ref{thm-ukl1}]
First, we claim that the upper bound $\min(m,d-2i)$ of summation in \eqref{ukl1} can be replaced with $m$, namely
\begin{align}\label{ukl1-another form}
c_{m,d}^i=\sum_{h=1}^{m}&
\frac{(e-i-h+1)(m+d)!}{e(e+1)(i+h)(i+h-1)(e-i)!(h-1)!i!(i-1)!},
\end{align}
where  $e=m+d-i-h$.
If  $m\leq d-2i$, then $\min(m,d-2i)=m$, and the claim is clearly true.
For $m\geq d-2i+1$, it suffices to show that
\begin{align*}
\sum_{h=d-2i+1}^{m}&
\frac{(e-i-h+1)(m+d)!}{e(e+1)(i+h)(i+h-1)(e-i)!(h-1)!i!(i-1)!}=0,
\end{align*}
or equivalently,
\begin{align*}
\sum_{h=d-2i+1}^{m}&
\frac{(e-i-h+1)}{e(e+1)(i+h)(i+h-1)(e-i)!(h-1)!}=0,
\end{align*}
where  $e=m+d-i-h$.
Letting $s=m+(d-2i+1)$, the above identity becomes
\begin{align*}
\sum_{h=d-2i+1}^{m} \frac{((s-h) -h)}{(i+s-h)(i+h) \times (i-1+(s-h))(i-1+h)\times (-1+s-h)!(-1+h)!}=0,
\end{align*}
which is obviously true since the interchange of $h$ with $s-h$ just changes the summand to its opposite value.

Note that \eqref{ukl1-another form} can be rewritten as
\begin{align*}
c_{m,d}^i
&=\binom{d+m}{i} \sum_{h=1}^{m}
\frac{(e-i-h+1)(m+d-i)!}{e(e+1)(i+h)(i+h-1)(e-i)!(h-1)!(i-1)!},
\end{align*}
where  $e=m+d-i-h$.
Denote by $f_{m,d}^i$ the summation on the right hand side, namely $f_{m,d}^i=c_{m,d}^i/\binom{d+m}{i}$.
By \eqref{ukl3} it suffices to show that
\begin{align*}
f_{m,d}^i={\frac{1}{d-i} {\sum _{h=0}^{m-1}\binom{d-i+h}{h+i+1}\binom{i-1+h}{h} }},
\end{align*}
or equivalently
\begin{align}\label{eq-mid1}
f_{m+1,d}^i-f_{m,d}^i=\frac{1}{d-i}\binom{d-i+m}{i+m+1}\binom{i+m-1}{m}
\end{align}
with
\begin{align}\label{eq-mid2}
f_{1,d}^i=\frac{1}{d-i}\binom{d-i}{i+1}.
\end{align}
The latter \eqref{eq-mid2} is obvious and the former  \eqref{eq-mid1} can be proved by using Zeilberger's algorithm along the following lines:

\begin{mma}
\In e=m+d-i-h;\\
\end{mma}
\begin{mma}
\In |Zb|[\frac{(e-i-h+1)(m+d-i)!}{e(e+1)(i+h)(i+h-1)(e-i)!(h-1)!(i-1)!},\{h,1,m\},m,1]\\
\end{mma}
\begin{mma}
\In |FunctionExpand|[\%]/.|Gamma|[n\_]\to(n-1)!\\
\Out \left\{-|SUM|[m]+|SUM|[1+m]==\frac{(d-i+m)!}{(d-i) (i+m) (i+m+1) (i-1)! m! (d-2 i-1)!} \right\}\\
\end{mma}

As indicated above, we have
\begin{align*}
f_{m+1,d}^i-f_{m,d}^i&=\frac{(d-i+m)!}{(d-i) (i+m) (i+m+1) (i-1)! m! (d-2 i-1)!}\\
&=\frac{1}{d-i}\binom{d-i+m}{i+m+1}\binom{i+m-1}{m},
\end{align*}
as desired. This completes the proof.
\end{proof}

\section{The \texorpdfstring{$Z$}{Z}-polynomials}\label{sec:uz}

The aim of this section is to prove Theorems \ref{thm-uz1} and \ref{thm-uzroot} based on Theorem \ref{thm-ukl2}.

\subsection{Polynomial coefficients}
In this subsection we will give a proof of Theorems \ref{thm-uz1}. Before that, let us first prove the following result.
\begin{thm}\label{uz2}
For any $m,d$, and $0\leq i \leq d-1$, we have
\begin{align}
z_{m,d}^i=\frac{\binom{d+2 m}{i+m}\binom{d }{i}}{\binom{d+2 m}{m}}\sum_{h=1}^{m}\frac{(-1)^{h+1} h }{m} \binom{i+m}{m-h} \binom{d-i-h+m-1}{m-1}
\end{align}
with $z_{m,d}^d=1$.
\end{thm}
\begin{proof}
By equating coefficients on both sides of \eqref{kltoz},
we find that $z_{m,d}^d=1$ and
\begin{align*}
 z_{m,d}^{d-i}=\sum_{k=i}^{2i-1}{\binom{d+m}{k+m}  c_{m,k}^{k-i}}=\sum_{k=0}^{i-1}{\binom{d+m}{k+i+m}  c_{m,k+i}^{k}},
\end{align*}
for each  $1 \leq i\leq d$.
Then substituting  \eqref{ukl2}  into the above identity  yields
\begin{align*}
z_{m,d}^{d-i}&=\sum_{k=0}^{i-1}\binom{d+m}{k+i+m} \binom{k+i+m}{k}\sum_{h=1}^{m}\frac{(-1)^{h+1}h}{m+i-h} \binom{m+i-h}{i-k-h}\binom{m+k}{m-h}.
\end{align*}

By interchanging the order of summation, we have
\begin{align*}
z_{m,d}^{d-i}&=\sum _{h=1}^{m} \frac{(-1)^{h+1}h}{m+i-h} \sum_{k=0}^{i-1}{ \binom{d+m}{k+i+m}\binom{k+i+m}{k} \binom{m+i-h}{i-k-h}\binom{m+k}{m-h}}.
\end{align*}
Note that
\begin{align*}
\binom{d+m}{k+i+m}\binom{k+i+m}{k}&=\binom{d+m}{i+m}\binom{d-i}{k},\\[5pt]
\binom{m+i-h}{i-k-h}\binom{m+k}{m-h}&=\binom{i}{i-h-k}\binom{m+i-h}{m-h}.
\end{align*}
Thus, we have
\begin{align*}
z_{m,d}^{d-i}
&=\binom{d+m}{i+m}\sum _{h=1}^{m}{\frac{(-1)^{h+1}h}{m+i-h} \binom{m+i-h}{m-h}
\sum_{k=0}^{i-1}\binom{i}{i-h-k} \binom{d-i}{k}  }\\
&=\binom{d+m}{i+m}\sum _{h=1}^{m}{\frac{(-1)^{h+1}h}{m+i-h} \binom{m+i-h}{m-h}
\sum_{k=0}^{i-h}\binom{i}{i-h-k} \binom{d-i}{k}}
\end{align*}
where the last equality holds since $\binom{i}{i-h-k}=0$ for $k>i-h$. By the Chu-Vandermonde identity, for each  $1 \leq i\leq d$ we have
\begin{align*}
z_{m,d}^{d-i}&=\binom{d+m}{i+m}\sum _{h=1}^{m}{\frac{(-1)^{h+1}h}{m+i-h} \binom{m+i-h}{m-h}
\binom{d}{i-h}}.
\end{align*}
Further replacing $i$ with $d-i$, we get that
\begin{align*}
z_{m,d}^i&=\binom{d+m}{i}\sum_{h=1}^{m}{\frac{(-1)^{h+1}h}{m+d-i-h} \binom{m+d-i-h}{m-h}\binom{d}{i+h}}\\
&=\frac{\binom{d+2 m}{i+m}\binom{d }{i}}{\binom{d+2 m}{m}}\sum_{h=1}^{m}\frac{(-1)^{h+1} h }{m} \binom{i+m}{m-h} \binom{d-i-h+m-1}{m-1},
\end{align*}
where $0\leq i \leq d-1$. This completes the proof. 
 \end{proof}

We proceed to prove Theorem \ref{thm-uz1}.

\begin{proof}[Proof of Theorem \ref{thm-uz1}]
Substituting $i=d$ into the right hand side of \eqref{uz1}, we get that
\begin{align*}
z_{m,d}^{d}&=\sum_{h=0}^{m-1}\frac{d (h-m+1)+m}{(h+1) m} \binom{d-1+h}{h}\\
&=\sum_{h=0}^{m-1}\left(\frac{d}{m}-\frac{d-1}{h+1} \right) \binom{d-1+h}{h}\\
&=\sum _{h=0}^{m-1} \left(\frac{d }{m}\binom{d+h-1}{h}-\binom{d+h-1}{h+1}\right)\\
&=\frac{d }{m}\sum_{h=0}^{m-1}\left(\binom{d+h}{h}-\binom{d+h-1}{h-1}\right)  -\sum _{h=0}^{m-1} \left(\binom{d+h}{h+1}-\binom{d+h-1}{h}\right)\\
&=\frac{d }{m} \binom{d+m-1}{m-1}-\left(\binom{d+m-1}{m}-1\right)\\
&=1.
\end{align*}
This shows that \eqref{uz1}
  holds for $i=d$ in view of Theorem \ref{uz2}.

Next we show that \eqref{uz1} holds for $0\leq i<d-1$.
By \eqref{uz1} and  \eqref{uz2}, it is sufficient to show that
$f_{m,d}^{i}=g_{m,d}^{i}$, where
\begin{align*}
f_{m,d}^{i}&=\sum_{h=1}^{m}\frac{(-1)^{h+1} h }{m} \binom{i+m}{m-h} \binom{d-i-h+m-1}{m-1},\\
g_{m,d}^{i}&=\sum _{h=0}^{m-1}  \frac{i (h-m+1)+m}{(h+1) m} \binom{i-1+h}{h} \binom{d-i+h}{h}.
\end{align*}
It is clear that $f_{1,d}^{i}=g_{1,d}^{i}=1$.
It remains to show that $f_{m,d}^{i}$ and $g_{m,d}^{i}$ have the same recurrence relation with respect to $m$.
We first use Zeilberger's algorithm to determine the recursion satisfied by $f_{m,d}^{i}$.

\begin{mma}
\In |Zb|\big[\frac{(-1)^{h+1} h }{m} |Binomial|[i+m,m-h] |Binomial|[d-i-h+m-1,m-1],\{h, 1, m\},m \break \big];\\
\end{mma}
\begin{mma}
\In |FunctionExpand|[\%]/.|Gamma|[n\_]\to(n-1)!\\
\end{mma}
\begin{mma}
\Out \big\{-m (d + d i - i^2 + d m)|SUM|[m]+ (1 + m) (d i - i^2 + d m)|SUM|[1+m]  =\frac{(i+m)! (d-i+m)!}{(i-1)! (m!)^2 (d-i-1)!}\big\} \\
\end{mma}

As indicated above, we have
\begin{align*}
(1+m) (d i-i^2+d m)f_{m+1,d}^{i}-m (d+d i-i^2+d m) f_{m,d}^{i}
 &=\frac{(i+m)! (d-i+m)!}{(i-1)! (m!)^2 (d-i-1)!}.
\end{align*}

Similarly, we apply Zeilberger's algorithm to $g_{m,d}^{i}$ as follows.

\begin{mma}
\In |Zb|[\frac{i (h-m+1)+m}{(h+1) m}   |Binomial|[i-1+h,h] |Binomial|[d-i+h,h] ,\{h,0,m-1\},m];\\
\end{mma}
\begin{mma}
\In |FunctionExpand|[\%]/.|Gamma|[n\_]\to(n-1)!\\
\end{mma}
\begin{mma}
\Out \big\{-m (d + d i - i^2 + d m)|SUM|[m]+ (1 + m) (d i - i^2 + d m)|SUM|[1+m]  =\frac{(i+m)! (d-i+m)!}{(i-1)! (m!)^2 (d-i-1)!}\big\} \\
\end{mma}

We find that $g_{m,d}^{i}$ satisfies the same recurrence relation as $f_{m,d}^{i}$, as desired. This completes the proof.
\end{proof}

\subsection{Real zeros}\label{sec:uzroot}

This subsection is devoted to the study of the real-rootedness of the $Z$-polynomials $Z_{U_{m,d}}(t)$. As will be shown below,  Theorem \ref{thm-uzroot} would follow in the same manner as Theorem \ref{thm-uklroot}.

By Theorem \ref{thm-uz1}, we see that
\begin{align}\label{eq-uz1-reform}
Z_{U_{m,d}}(t)&=\frac{1}{\binom{d+2 m}{m}}\sum_{i=0}^{d}\binom{d+2 m}{i+m}\binom{d }{i} b_m(d,i) t^i ,
\end{align}
where
\begin{align}\label{eq-zpoldi}
b_m(d,i)=\sum _{h=0}^{m-1}  \frac{i (h-m+1)+m}{(h+1) m} \binom{i-1+h}{h} \binom{d-i+h}{h}.
\end{align}
Note that $i$ can take any integer value between $0$ and $d$ in the above formula, namely, $b_{m}(d,i)$ is well defined for $0\leq i\leq d$. Moreover, we have $b_{m}(d,d)=1$ since $z_{m,d}^d=1$.

Parallel to Lemma \ref{lem-uklroots}, we need the  following result to prove Theorem \ref{thm-uzroot}. We omit the proof of the lemma here.

\begin{lem}\label{lem-uzroots}
For any  $m$ and $d$, the polynomial
\begin{align}
\frac{1}{\binom{d+2 m}{m}}\sum_{i=0}^{d} {
\binom{d+2 m}{i+m}\binom{d}{i}  t^i},
\end{align}
has only real zeros.
\end{lem}

Parallel to Conjecture \ref{conj-ukl}, we have the following conjecture.

\begin{conj}\label{conj-uz}
 For any $m,d$, and $0\leq i\leq d$, let $b_{m}(d,i)$ be defined as in \eqref{eq-zpoldi}. Then $\{b_{m}(d,i)\}_{i=0}^{d}$ is a $d$-sequence.
\end{conj}

We can not prove this conjecture for any $m$. To prove Theorem \ref{thm-uzroot}, we only need to prove the following theorem.

\begin{thm}\label{thm-dsequence-z}
For any $d$ and $2\leq m\leq 15$, the sequence $\{b_{m}(d,i)\}_{i=0}^{d}$ is a $d$-sequence.
\end{thm}

\begin{proof}
The proof is similar to that of Theorem \ref{thm-dsequence}.
In fact, $f_{m}(d,i)$ and $b_{m}(d,i)$ share some common properties. Note that $b_{m}(d,i)$ can also be considered as a polynomial in $i$ and $d$ for any fixed $m$, and
as a polynomial in $i$ the polynomial $b_{m}(d,i)$ is of degree
$2(m-1)$ with leading coefficients $\frac{(-1)^{m-1}}{(m-1)! m!}$ and constant term $1$. Thus, we can express $b_{m}(d,i)$ in terms of the falling factorials as done for $f_{m}(d,i)$:
\begin{align}\label{eq-fg-z}
b_{m}(d,i)=\sum_{k=0}^{2(m-1)} y_{m,k}(d) (i)_k.
\end{align}

To prove the theorem, it suffices to show that
\begin{align}\label{eq-realroots-z}
R_d(t)=\sum_{i=0}^{d} b_{m}(d,i) \binom{d}{i} t^i
\end{align}
has only nonnegative zeros for $2\leq m\leq 15$.
Recall that in the proof of Theorem \ref{thm-dsequence} we transform the problem of proving the real-rootedness of
$Q_d(t)$ into the problem of proving the real-rootedness of
$G_{m,d}(t)$ defined by \eqref{eq-gdt}. In the same manner, we are able to reduce the problem of
proving the real-rootedness of
$R_d(t)$ into the problem of proving the real-rootedness of
the polynomial
\begin{align}\label{eq-gdt-z}
Y_{m,d}(t)=\sum_{k=0}^{2(m-1)} y_{m,k}(d) (d)_k t^k,
\end{align}
where $y_{m,k}(d)$ is defined by \eqref{eq-fg-z}.

The real-rootedness of $Y_{m,d}(t)$ can be proved along the lines of proving the real-rootedness of $G_{m,d}(t)$.
We still take $m=2$ to illustrate our proof.
We first expand $b_{m}(d,i)$ in terms of the falling factorials, and then determine the explicit expression of $Y_{m,d}(t)$.
For $m=2$ we have
\begin{align*}
b_{m}(d,i)&=
1+\frac{di}{2}-\frac{i^2}{2}=(i)_0+\frac{d-1}{2}(i)_1-\frac{1}{2}(i)_2\\
Y_{m,d}(t)&=1+\frac{d-1}{2} (d)_1t-\frac{1}{2}(d)_2t^2=1+\frac{d(d-1)}{2}t-\frac{d(d-1)}{2}t^2.
\end{align*}
Now, to prove the real-rootedness of $Y_{m,d}(t)$, there are two cases to consider: (i) $d<2(m-1)$; (ii) $d\geq 2(m-1)$.

Let us first consider the case of $d<2(m-1)$.
For $m=2$, such $d$ can only be $1$. In this case, we have
$Y_{2,1}(t)=1$.

We proceed to consider the case of $d\geq 2(m-1)$.
In this case, $Y_{m,d}(t)$ is obviously  a polynomial in $t$ of degree $2(m-1)$.
As before, by Theorem \ref{rz-criterion}, we need to prove the positivity of $2(m-1)$ Hurwitz determinants $\Delta_{2k}(Y_{m,d}(t),Y'_{m,d}(t))$ for $1\leq k\leq 2(m-1)$.
We find that the substitution of $d'=d-2(m-1)\geq 0$ is still helpful. In fact, we can show that for any $1\leq k\leq 2(m-1)$  the Hurwitz determinant $\Delta_{2k}(Y_{m,d}(t),Y'_{m,d}(t))$ is a polynomial in $d'$ with positive coefficients when $m$ is not too big. For $m=2$, it is straightforward to compute that
\begin{align*}
\Delta_{2}(Y_{2,d}(t),Y'_{2,d}(t))&=\det\left(
\begin{array}{cc}
 -\frac{1}{2} (d-1) d & \frac{1}{2} d (d-1) \\
 0 & -(d-1) d \\
\end{array}
\right)\\
&=\frac{1}{2} (d-1)^2 d^2\\
&=\frac{d'^4}{2}+3 d'^3+\frac{13 d'^2}{2}+6 d'+2,
\end{align*}
and
\begin{align*}\Delta_{4}(Y_{2,d}(t),Y'_{2,d}(t))&=\det
\left(
\begin{array}{cccc}
 -\frac{1}{2} (d-1) d & \frac{1}{2} (d-1) d & 1 & 0 \\
 0 & -(d-1) d & \frac{1}{2} (d-1) d & 0 \\
 0 & -\frac{1}{2} (d-1) d & \frac{1}{2} (d-1) d & 1 \\
 0 & 0 & -(d-1) d & \frac{1}{2} (d-1) d \\
\end{array}
\right)\\
&=\frac{1}{16} (d-1)^2 d^2 \left(d^4-2 d^3+9 d^2-8 d\right)\\
&=\frac{d'^8}{16}+\frac{3 d'^7}{4}+\frac{35 d'^6}{8}+\frac{63 d'^5}{4}+\frac{585 d'^4}{16}+54 d'^3+\frac{97 d'^2}{2}+24 d'+5,
\end{align*}
where $d'=d-2\geq 0$.
 For $3\leq m \leq 15$, a similar calculation can be found in \linebreak
\href{https://github.com/mathxie/kl_uniform_matroid}{https://github.com/mathxie/kl\_uniform\_matroid}.
This completes the proof. 
\end{proof}

Finally, we are in the position to prove Theorem \ref{thm-uzroot}.

\begin{proof}[Proof of Theorem \ref{thm-uzroot}.]
It is clear that $Z_{U_{m,d}}(t)$ is a polynomial with positive coefficients, see \eqref{kltoz}.
Combining Theorem \ref{thm-dsequence-z} and Lemma \ref{lem-uzroots} we obtain the desired result.
\end{proof}


 \section*{Acknowledgments}
The first author is supported in part by the Fundamental Research Funds for the Central Universities 3102017OQD101. The second author is supported in part by the National Science Foundation of USA
 grant DMS-1600811.
 This fourth author is supported in part by the National Science Foundation of China (Nos. 11231004, 11522110). 
The fifth author is supported in part by the National Science Foundation of China (Nos. 11626172, 11701424).



\bibliographystyle{abbrv}
\end{document}